\documentclass[12pt,a4paper]{article}
\usepackage{amsfonts,amsmath,amssymb}
\usepackage{oldlfont}
\usepackage[cp1250]{inputenc}
\usepackage[T1]{fontenc}
\usepackage[width=17cm,height=25cm]{geometry}

\usepackage{color}

\usepackage{pgf,tikz}
\usetikzlibrary{arrows,patterns}
\usetikzlibrary{matrix,positioning}

\newtheorem{theorem}{\bf Theorem}[section]
\newtheorem{corollary}[theorem]{Corollary}
\newtheorem{lemma}[theorem]{Lemma}

\newcommand{\qed}{\hfill $\square$ \bigskip}
\newcommand{\bt}{\boxtimes}

\begin{document}

\title{Toll number of the strong product of graphs}

\author{Tanja Gologranc $^{a,b}$ \and
Polona Repolusk $^{a}$}

\maketitle

\begin{center}
$^a$ Faculty of Natural Sciences and Mathematics, University of Maribor, Slovenia\\
{\tt tanja.gologranc1@um.si\\
polona.repolusk@um.si}\\
\medskip

$^b$ Institute of Mathematics, Physics and Mechanics, Ljubljana, Slovenia\\
\end{center}

\begin{abstract}
A tolled walk $T$ between two non-adjacent vertices $u$ and $v$ in a graph $G$ is a walk, in which $u$ is adjacent only to the second vertex of $T$ and $v$ is adjacent only to the second-to-last vertex of $T$. A toll interval between $u,v\in V(G)$ is a set $T_G(u,v)=\{x\in V(G)~|~x \textrm{ lies on a tolled walk between } u \textrm{\, and\,} v\}$. A set $S \subseteq V(G)$ is toll convex, if $T_{G}(u,v)\subseteq S$ for all $u,v\in S$. A toll closure of a set $S \subseteq V(G)$ is the union of toll intervals between all pairs of vertices from $S$. The size of a smallest set $S$ whose toll closure is the whole vertex set is called a toll number of a graph $G$, $tn(G)$. 
This paper investigates the toll number of the strong product of graphs.  First, a description of  toll intervals between two vertices in the strong product graphs is given. Using this result we characterize graphs with $tn(G \bt H)=2$ and graphs with $tn(G \bt H)=3$, which are the only two possibilities. As an addition, for the t-hull number of $G\bt H $ we show that  $th(G \boxtimes H) = 2$ for any non complete graphs $G$ and $H$. As extreme vertices play an important role in different convexity types, we show that no vertex of the strong product graph of two non complete graphs is an extreme vertex with respect to the toll convexity.



\end{abstract}

\newpage

\section{Introduction}\label{Int}

The classical theory of convexity is based on three natural conditions, imposed on a family of subsets of a given set. All three axioms hold in the interval convexity, which was emphasized in~\cite{vel-93} as one of the most natural ways for introducing convexity. An interval $I:X \times X \rightarrow 2^X$ has the property that $x,y\in I(x,y)$, and convex sets are defined as the sets $S$ in which all intervals between elements from $S$ lie in $S$. In terms of graph theory, several interval structures have been introduced. The interval function $I$ is most commonly defined by a set of paths between two vertices, where these paths have some interesting properties. For instance, shortest paths yield geodesic intervals, induced paths yield monophonic intervals etc. Each type of an interval then gives rise to the corresponding convexity, see~\cite{chmusi-05,Pela} for some basic types of intervals/convexities. 

A type of a graph convexity, called toll convexity, was introduced in \cite{abg}. The definition of a tolled walk generalizes the one of monophonic and geodesic paths, as any geodesic path is also a monophonic path, and any monophonic path is also a tolled walk.  Authors of the first paper on the topic \cite{abg} focused their attention on interval graphs and applied the concept from~\cite{al-14}, where this family of graphs was characterized in terms of tolled walks. They used tolled walks from~\cite{al-14} to define a type of convexity for which exactly interval graphs are convex geometries (convex sets can be build from its extreme elements). The same property is known also for other types of convexities. In the case of monophonic convexity, exactly
chordal graphs are convex geometries, while in the geodesic convexity, these are precisely Ptolemaic graphs (i.e. distancehereditary chordal graphs), see~\cite{fj-86}.
 
Authors also considered other properties of the toll convexity followed by results, already known for other types of convexities. They investigated the toll number and the t-hull number of a graph, which were investigated in terms of the geodesic convexity about 30 years ago~\cite{Evertt85,hlt} and intensively studied after that, for instance in  graph products~\cite{bkh, bkt,CC}, in terms of other types of convexities~\cite{cz,stg} and more. See~\cite{CHZ02,cpz,Pela} for further reading on this topic. 

The toll number was investigated in terms of the Cartesian and the lexicographic product of graphs. t-convex sets of these two products were described in \cite{abg,gr} and in \cite{gr} it was shown that $tn(G \Box H)=2$, and if $H$ is not isomorphic to a complete graph, $tn(G \circ H) \leq 3\cdot tn(G)$. There is also an exact formula for $tn(G \circ H)$ - it is described in terms of the so-called toll-dominating triples.

In terms of the strong product of graphs, the geodetic and the hull number were investigated in \cite{strong1,strong2,strong3}, where bounds for the geodetic and hull number of the strong product of graphs were found, and some exact values for different families of graphs, which include strong products of paths and complete graphs with a cycle, products of complete and complete bipartite graphs. In \cite{abg}, t-convex sets of $G \bt H$ were characterized. Our paper takes into consideration the toll number of the strong product of graphs. We proceed as follows.

First we list all necessary definitions for our work. Section \ref{interval} is devoted to describing vertices that lie on a toll interval between two vertices in $G \bt H$. It turns out that if  $x_1x_2 \notin E(G), y_1y_2\notin E(H)$, the toll interval between $(x_1,y_1),(x_2,y_2) \in V(G \bt H)$ contains all vertices of $V(G \bt H)$, except maybe some neighbors of $(x_1,y_1),(x_2,y_2)$. The last part of this section focuses on extreme vertices with respect to the toll convexity. We show that if $G$ and $H$ are not complete graphs, then a graph $G \boxtimes H$ has no extreme vertices. Using results of this section helps build results of Section \ref{s:tn}, where a complete description of the toll number of the strong product of graphs is given. We show that $tn(G\bt H) \leq 3$ and characterize graphs with  $tn(G\bt H) = 2$. Consequently, we also get a characterization of those with  $tn(G\bt H) = 3$. In Concluding remarks we finish with quite a straightforward result concerning t-hull number of the strong product, $th(G \boxtimes H) = 2$ for non complete graphs $G$ and $H$.

\section{Preliminaries}\label{Pre}

All graphs, considered in this paper, are finite, simple, non-trivial (i.e. graphs with at least two vertices), connected and without multiple edges or loops.

Let $G=(V(G),E(G))$ be a graph. The {\it distance} $d_G(u,v)$ between vertices $u,v\in V(G)$ is the length of a shortest path between $u$ and $v$ in $G.$  
The  {\it diameter} of a graph, $diam(G)$, is defined as  $\displaystyle diam(G)=\max\lbrace d_G(u,v)~|~ u,v\in V(G)\rbrace$.
Sometimes we will call vertices $u$ and $v$, for which $d_G(u,v)=diam(G)$, diametral vertices of a graph $G$.
The  {\it eccentricity} of a vertex $v \in V(G)$, $e_G(v)$, is defined as  $ e_G(v)=\max \lbrace d_G(v,x) ~|~ x \in V(G) \rbrace$. A vertex $x$ is said to be eccentric with respect to a vertex $v$, if $d_G(x,v)=e_G(v)$. The set of all eccentric vertices of a vertex $v$ is denoted with $ecc_G(v)$.

The {\it{geodesic interval}} $I_{G}(u,v)$ between vertices $u$ and $v$ is the set of all vertices that lie on some shortest path between $u$ and $v$ in $G$, i.e. $I_{G}(u,v)=\{x\in V(G) ~|~ d_G(u,x)+d_G(x,v)=d_G(u,v)\}$. A subset $S$ of $V(G)$ is {\it{geodesically convex}} (or {\it g-convex}) if $I_{G}(u,v)\subseteq S$ for all $u,v\in S$. Let $S$ be a set of vertices of a graph $G$. Then the geodetic closure $I_G \left[ S \right]$ is the union of geodesic intervals between all pairs of vertices from $S$, that is, $I_G \left[ S \right] = \bigcup_{u,v \in S}I_G(u,v)$. A set $S$ of vertices of $G$ is a {\it geodetic set} in $G$ if $I_G \left[ S \right] = V (G)$. The size of a minimum geodetic set in a graph $G$ is called the {\it geodetic number} of $G$ and denoted by $g(G)$. Given a subset $S \subseteq V(G)$, the {\it convex hull} $\left[ S \right]$ of $S$ is the smallest convex set that contains $S$. We say that $S$ is a {\it hull set} of $G$ if $\left[ S \right]=V(G)$. The size of a minimum hull set of $G$ is the {\it hull number} of $G$, denoted $hn(G)$. Indices above may be omitted, whenever the graph $G$ is clear from the context.

All definitions, listed above for the geodesic convexity, could be rewritten in terms of monophonic convexity, all-path convexity, Steiner convexity etc. For more details see surveys~\cite{bkt-10,cpz}, the book~\cite{Pela} and the paper~\cite{Steiner-spa}. In the rest of the paper, the  term convexity will always stand for the so-called toll convexity, unless we will say otherwise.

Let $u$ and $v$ be two different, non-adjacent vertices of a graph $G$. A {\it tolled walk} $T$ between $u$ and $v$ in $G$ is a sequence of vertices of the form $T: u,w_1,\ldots,w_k,v,$ where $k\ge 1$, which enjoys the following three conditions:
\begin{itemize}
\item $w_iw_{i+1}\in E(G)$ for all $i$,
\item $uw_i\in E(G)$ if and only if $i=1$,
\item $vw_i\in E(G)$ if and only if $i=k$.
\end{itemize}
In other words, a tolled walk is any walk between $u$ and $v$ such that $u$ is adjacent only to the second vertex of the walk and $v$ is adjacent only to the second-to-last vertex of the walk. 
For $uv\in E(G)$ let $T:u,v$ be a tolled walk as well. The only tolled walk that starts and ends in the same vertex $v$ is $v$ itself. 
We define $T_G(u,v)=\{x\in V(G) ~|~ x \textrm{ lies on a tolled walk between } u \textrm{\, and\,} v\}$ to be the {\it toll interval} between $u$ and $v$ in $G$. Finally, a subset $S$ of $V(G)$ is {\it{toll convex}} (or {\it t-convex}) if $T_{G}(u,v)\subseteq S$ for all $u,v\in S$. The {\it toll closure} $T_G[S]$ of a subset $S\subseteq V(G)$ is the union of toll intervals between all pairs of vertices from $S$, i.e. $T_G[S]=\bigcup_{u,v\in S} T_G(u,v)$. If $T_G[S]=V(G)$, we call $S$ a {\it toll set} of a graph $G$. The size of a minimum toll set in $G$ is called the {\it toll number} of $G$ and is denoted by $tn(G)$. Again, when the graph is clear from the context, indices may be omitted.

A {\it t-convex hull} of a set $S\subseteq V(G)$ is defined as the intersection of all t-convex sets that contain $S$ and is denoted by $[S]_t$. A set $S$ is a {\it t-hull set} of $G$ if its t-convex hull $[S]_t$ coincides with $V(G)$. The {\it t-hull number} of $G$ is the size of a minimum t-hull set and is denoted by $th(G)$.
Given the toll interval $T_G \,:\, V \times V \rightarrow 2^V$ and a set $S\subset V(G)$ we define $T_G^k[S]$ as follows: $T_G^0[S]=S$ and $T_G^{k+1}[S]=T_G[T_G^k[S]]$ for any $k\geq 1$. Note that $[S]_t= \bigcup_{k \in \mathbb{N}} T_G^k[S]$.  From definitions above we immediately infer that every toll set is a t-hull set, and hence $th(G)\leq tn(G)$ for any graph $G$.

A vertex $s$ from a convex set $S$ in a graph $G$ is an {\em extreme vertex} of $S$, if $S-\{s\}$ is also convex in $G$. Thus, also extreme vertices can be defined in terms of different graph convexities. Considering the geodesic and the monophonic convexity, the extreme vertices are exactly simplicial vertices, i.e.\ vertices, whose closed neighborhoods induce a complete graph. 
For the toll convexity, any extreme vertex is also a simplicial, but the converse is not necessarily true, see~\cite{abg}. The set of all extreme vertices of a graph $G$, denoted  $Ext(G)$, is contained in any toll set of $G$. Even more, it is contained in any t-hull set of $G$, i.e.\ $|Ext(G)| \leq th(G) \leq tn(G)$.
The assertion holds also in other types of convexities~\cite{CHZ02,Evertt85}.

We follow the definitions of graph products by the book \cite{ImKl}. The vertex set of the strong product $G \bt H$ of graphs $G$ and $H$ is equal to $V(G)\times V(H)$. Vertices $(g_{1},h_{1})$ and $(g_{2},h_{2})$ are adjacent in $G \bt H$  if either ($g_1g_2\in E(G)$ and $h_1=h_2$) or ($g_1=g_2$ and $h_1h_2\in E(H)$) or ($g_1g_2 \in E(G)$ and $h_1h_2 \in E(H)$). For a vertex $h\in V(H)$, we call the set $G^{h}=\{(g,h)\in V(G \bt H)~|~g\in V(G)\}$ a $G$-\emph{layer} of $G \bt H$. By abuse of notation we also consider $G^{h}$ as the corresponding induced subgraph. Clearly $G^{h}$ is isomorphic to $G$. For $g\in V(G)$, the $H$-\emph{layer} $^g\!H$ is defined as $^g\!H =\{(g,h)\in V(G \bt H) ~|~ h\in V(H)\}$. We also consider $^g\!H$ as an induced subgraph and note that it is isomorphic to $H$. 
For more details on graph products see~\cite{ImKl}.

\section{Toll intervals in $G \bt H$}\label{interval}

This section will be devoted to describe toll intervals between two vertices in the strong product graphs. More precisely, we will prove that if  $x_1x_2 \notin E(G), y_1y_2\notin E(H)$, the toll interval between $(x_1,y_1),(x_2,y_2) \in V(G \bt H)$ contains all vertices of $V(G \bt H)$, except maybe some neighbors of $(x_1,y_1),(x_2,y_2)$. For the sake of clarity we divided this investigation into several lemmas. The final result will play an important role in the rest of the paper.

\begin{lemma}\label{l:main1}
Let $(x_1,y_1),(x_2,y_2) \in V(G \boxtimes H)$, where $x_1x_2\notin{E(G)}$, $y_1y_2 \notin E(H)$. If $x \in T_G(x_1,x_2)$ and $y \in T_H(y_1,y_2)$, then $(x,y) \in T_{G \boxtimes H}\left((x_1,y_1),(x_2,y_2)\right)$.
\end{lemma}

\begin{proof}
Let $P$ be an $x_1,x_2$-tolled walk that contains $x$ and $Q$ a $y_1,y_2$-tolled walk that contains $y$. Denote also $P_1$ to be a subwalk of $P$ from $x_1$ to $x$ and $P_2$ a subwalk of $P$ from $x$ to $x_2$. 

Assume first that $xx_1 \notin E(G)$ and $xx_2 \notin E(G)$. Then, by concatenating walks $P_1$ in the layer $G^{y_1}$, $Q$ in the layer $^x\!H$ layer and $P_2$ in the layer $G^{y_2}$ we get an $(x_1,y_1),(x_2,y_2)$-tolled walk that contains $(x,y)$.  

Denote for a moment $z_1$ to be a vertex on $Q$, that is adjacent to $y_1$ and $z_2$ a vertex on $Q$, that is adjacent to $y_2$. If $x_1x \in E(G)$ and $y \notin \lbrace y_1,y_2 \rbrace$, then the walk, described above, must be shortened in a way that we replace a subwalk $(x_1,y_1),\ldots,(x,y_1),(x,z_1)$ with  $(x_1,y_1)(x,z_1)$. This then gives a desired tolled walk. If $xx_2 \in E(G)$, then a similar replacement of a walk $(x,z_2),(x,y_2),\ldots ,(x_2,y_2)$ with $(x,z_2)(x_2,y_2)$ gives a tolled walk containing the vertex $(x,y)$. 

Now let  $y =y_1$. Note that $yy_2 \notin E(G)$. Then a concatenation of a walk $P$ in the layer $G^{y}$ together with a walk $Q$ in the layer $^{x_2}\!H$ is a desired tolled walk. The proof for $y=y_2$ is similar. \qed
\end{proof}

\begin{lemma}\label{l:main2b}
Let $(x_1,y_1),(x_2,y_2) \in V(G \boxtimes H)$, where $x_1x_2\notin{E(G)}$, $y_1y_2 \notin E(H)$. Let $x \notin T_G(x_1,x_2)$ and $y \in T_H(y_1,y_2)$. If  $(x,y)\in V(G \boxtimes H)$ such that $(x,y) \notin N((x_1,y_1))  \cup N((x_2,y_2))$, then $(x,y) \in T_{G \boxtimes H}\left((x_1,y_1),(x_2,y_2)\right)$.
\end{lemma}

\begin{proof}
Let $P$ be a shortest $x_1,x_2$-walk that contains $x$ and $Q$ a $y_1,y_2$-tolled walk that contains $y$. As $P$ is not a tolled walk, there must be at least two vertices on $P$, adjacent either to $x_1$ or $x_2$. Without loss of generality, assume there are at least two vertices on $P$, adjacent to $x_1$.
Denote $P=x_1,p_0,\ldots,p_k,x,p_{k+1},\ldots,p_l,x_2$. If $x_1p_i\in E(G)$ for some $i \in \lbrace 0,\ldots,k\rbrace$, then $P$ could be shortened, a contradiction. Therefore,  if $x_1p_i\in E(G)$, then $i \in \lbrace k+1,\ldots, l \rbrace$. Also note that $p_i\neq x_2$ because $x_1x_2 \notin E(G)$.
Similar arguments lead to the fact that if $p_jx_2 \in E(G)$, then $j\in \lbrace 0,\ldots,k \rbrace$ and $p_j \neq x_1$.
Let us define the following walks:
\begin{itemize}
\item $W_1$: $x_1,p_0,\ldots,p_k,x$ in the layer $G^{y_1}$, 
\item $W_2$: $Q$ in the layer $^x\!H$,  and 
\item $W_3$: $x,p_{k+1},\ldots,p_l,x_2$ in the layer $G^{y_2}$.
\end{itemize}

As $x \notin T_G(x_1,x_2)$, $x_1x$ and $xx_2$ are not both edges in $G$. If $x_1x \notin E(G)$ and $xx_2 \notin E(G)$, then, by concatenating $W_1$, $W_2$ and $W_3$, we get a tolled $(x_1,y_1),(x_2,y_2)$-walk that contains $(x,y)$. 

Assume now that $x_1x \in E(G)$ (note that in this case $xx_2 \notin E(G)$). Since $(x,y)$ is not adjacent to $(x_1,y_1)$, $y \notin N_H[y_1]$. 
Let $z_1$ be a vertex on $Q$, that is adjacent to $y_1$ and $z_2$ a vertex on $Q$, that is adjacent to $y_2$. 
Then the walk $(x_1,y_1),(x,z_1)$, a subwalk of $W_2$ between $(x,z_1)$ and $(x,y_2)$, and the walk $W_3$ together form a desired tolled walk. Similarly, if $xx_2 \in E(G)$, then $y \notin N_H[y_2]$, as $(x,y)(x_2,y_2) \notin E(G\bt H)$. Then a walk $W_1$, a subwalk of $W_2$ between $(x,y_1)$ and $(x,z_2)$ and a walk $(x,z_2),(x_2,y_2)$ together form a desired tolled walk.\qed
\end{proof}

\begin{lemma}\label{l:main2}
Let $(x_1,y_1),(x_2,y_2) \in V(G \boxtimes H)$, where $x_1x_2\notin E(G)$, $y_1y_2 \notin E(H)$. If $x \notin T_G(x_1,x_2)$ and $y \in T_H(y_1,y_2)-\{y_1,y_2\}$, then $(x,y) \in T_{G \boxtimes H}\left((x_1,y_1),(x_2,y_2)\right)$.
\end{lemma}

\begin{proof}
If $(x,y) \notin N((x_1,y_1))  \cup N((x_2,y_2))$, then the result follows by Lemma~\ref{l:main2b}. Hence we may assume without loss of generality that $(x,y)$ is adjacent to $(x_1,y_1)$. 
Since $y \neq y_1$, $yy_1 \in E(H)$ and because $x \notin T_G(x_1,x_2)$, $x \neq x_1$. By the definition of adjacency in the strong product it follows that $xx_1 \in E(G)$.
Also, note that $xx_2 \notin E(G)$, otherwise $x_1,x,x_2$ would be a tolled $x_1,x_2$-walk that contains $x$. 

Let $P$ be a shortest $x,x_2$-path and $Q$ a $y_1,y_2$-tolled walk that contains $y$. Denote also $Q_1$ to be a subwalk of $Q$ between $y$ and $y_2$.  We construct a tolled $(x_1,y_1),(x_2,y_2)$-walk that contains $(x,y)$ as follows. Concatenate a walk $(x_1,y_1),(x,y)$ with $Q_1$ in $^{x}\!H$ and $P$ in $G^{y_2}$.  \qed 

\end{proof}

%
%

By the symmetry of the strong product we have the following two lemmas:

\begin{lemma}\label{l:main3b}
Let $(x_1,y_1),(x_2,y_2) \in V(G \boxtimes H)$, where $x_1x_2\notin{E(G)}$, $y_1y_2 \notin E(H)$. Let $x \in T_G(x_1,x_2)$ and $y \notin T_H(y_1,y_2)$. If  $(x,y)\in V(G \boxtimes H)$ such that $(x,y) \notin N((x_1,y_1))  \cup N((x_2,y_2))$, then $(x,y) \in T_{G \boxtimes H}\left((x_1,y_1),(x_2,y_2)\right)$.
\end{lemma}

\begin{lemma}\label{l:main3}
Let $(x_1,y_1),(x_2,y_2) \in V(G \boxtimes H)$, where $x_1x_2\notin{E(G)}$, $y_1y_2 \notin E(H)$. If $x \in T_G(x_1,x_2)-\{x_1,x_2\}$ and $y \notin T_H(y_1,y_2)$, then $(x,y) \in T_{G \boxtimes H}\left((x_1,y_1),(x_2,y_2)\right)$.
\end{lemma}


\begin{lemma}\label{l:main4}
Let $(x_1,y_1),(x_2,y_2) \in V(G \boxtimes H)$, where $x_1x_2\notin{E(G)}$, $y_1y_2 \notin E(H)$. Let $x \notin T_G(x_1,x_2)$ and $y \notin T_H(y_1,y_2)$. If  $(x,y)\in V(G \boxtimes H)$ such that $(x,y) \notin N((x_1,y_1))  \cup N((x_2,y_2))$, then $(x,y) \in T_{G \boxtimes H}\left((x_1,y_1),(x_2,y_2)\right)$.
\end{lemma}

\begin{proof}
Let $x \notin T_G(x_1,x_2)$ and $y \notin T_H(y_1,y_2)$. Note that therefore $x \notin \lbrace x_1,x_2\rbrace$ and $y \notin \lbrace y_1,y_2\rbrace$.
Denote $P=x_1,p_0,\ldots,p_k,x,p_{k+1},\ldots,p_l,x_2$ to be a shortest $x_1,x_2$-walk that contains $x$ and $Q=y_1,q_0,\ldots,q_a,y,q_{a+1},\ldots,q_b,y_2$ a shortest $y_1,y_2$-walk that contains $y$. 
Therefore there exist at least one of $\lbrace i_x,j_x\rbrace$ and at least one of $\lbrace i_y,j_y\rbrace$, such that:
\begin{itemize}
\item $x_1p_{i_x} \in E(G)$ and $i_x \in \lbrace k+1,\ldots,l\rbrace$,
\item $x_2p_{j_x} \in E(G)$ and $j_x \in \lbrace 0,\ldots,k\rbrace$,
\item $y_1q_{i_y} \in E(H)$ and $i_y \in \lbrace a+1,\ldots,b\rbrace$,
\item $y_2q_{j_y} \in E(H)$ and $j_y \in \lbrace 0,\ldots,a\rbrace$.
\end{itemize}
Similarly as in Lemma \ref{l:main2b}, note that $i_x \notin \lbrace 0,\ldots,k\rbrace$ because $P$ was chosen to be the shortest walk between $x_1$ and $x_2$, which passes through $x$. Similar holds for $i_y,j_x,j_y$.
Let $i_x,i_y$ be smallest indices and $j_x,j_y$ the biggest indices, such that the upper terms hold.

Assume first that $x_1x \notin  E(G)$ and assume that there exists $p_{i_x}$ as described above. Since $x \notin T_G(x_1,x_2)$,
 $xx_2 \notin E(G)$.  Let $W_1$ be a walk $x_1,p_{i_x},\ldots,p_{k+1},x$ in the layer $G^{y_1}$ and $W_2$ the walk $Q$ in the layer $^x\!H$. Note that no vertex of $^{x_1}\!H$ is adjacent to a vertex of $^x\!H$ and the same holds for vertices of $^{x_2}\!H$. Let finally $W_3$ be a walk $x,p_{k+1}\ldots,p_l,x_2$ in the layer $G^{y_2}$. As $y_1y_2 \notin E(G)$, no vertex of $W_3$ is adjacent to $(x_1,y_1)$. By concatenating walks $W_1,W_2$ and $W_3$ we get an $(x_1,y_1),(x_2,y_2)$-tolled walk $W$ that contains $(x,y)$. 

For the second case, assume that  $x_1x \notin  E(G)$ but there does not exist $p_{i_x}$ as described above. Therefore there must be $p_{j_x}$ such that $x_2p_{j_x} \in E(G)$ and $j_x \in \lbrace 0,\ldots,k\rbrace$. 
If $xx_2 \notin E(G)$, then a walk 
$W_1$,  which is $x_1,p_0,\ldots,p_k,x$ in the layer $G^{y_1}$,
$W_2$, which is $Q$ in the layer  $^{x}\!H$ and 
$W_3$, which is a walk $x, p_{k},\ldots,p_{j_x},x_2$ in the layer $G^{y_2}$ 
alltogether give a desired tolled walk. 
Assume now that $xx_2 \in E(G)$. Note that we still have $xx_1 \notin E(G)$ and there does not exist $p_{i_x}$ as described above. If also $yy_2 \in E(H)$, we have $(x,y)\in N\left((x_2,y_2)\right)$, which is not true by the assumption. Therefore $yy_2 \notin E(H)$.
Define 
$W_1$,  to be a walk $y_1,q_0,\ldots,q_a,y$ in the layer $^{x_1}\!H$,
$W_2$, which is $x_1,p_0,\ldots,p_k,x$ in the layer  $G^{y}$,
$W_3$, which is a walk $y,q_{a+1},\ldots,q_b$ in the layer $^{x}\!H$ and a walk $(x,q_b),(x_2,y_2)$. Note that no vertex of $W_1$ is adjacent to a vertex of $^{x_2}\!H$ and because $Q$ was the shortest walk between $y_1$ and $y_2$ that contains $y$, only one vertex of $W_1$ is adjacent to $(x_1,y_1)$. Note that in case $yy_1 \in E(H)$, $W_1$ must be shortened using diagonal edges, similarly as on many points in the previous proofs. As $yy_2 \notin E(H)$, no vertex of $W_2$ is adjacent to a vertex of  $G^{y_2}$. 
If there is $j_y$ such that $y_2q_{j_y} \in E(H)$, the only possibility is that $j_y \in \lbrace 0,\ldots,a\rbrace$. Therefore also $W_3$ meets the conditions to get a tolled walk between $(x_1,y_1)$ and $(x_2,y_2)$ that goes through $(x,y)$.


To complete the proof, assume that $xx_1 \in E(G)$. As $(x,y) \notin N\left((x_1,y_1)\right)$, $y_1y \notin E(H)$. 
Note also that there is $i_x \in \lbrace k+1,\ldots,l\rbrace$, such that $x_1p_{i_x} \in E(G)$. 
Since $x \notin T_G(x_1,x_2)$, $xx_2 \notin E(G)$. Let $W_1$ be a walk $(x_1,y_1),(x,q_0),(x,q_1),\ldots,(x,y)$. No vertex of $W_1$ (besides $(x,q_0)$) is adjacent to $(x_1,y_1)$ because $y_1q_i \in E(H)$ only if $i \in \lbrace a+1,\ldots,b\rbrace$, and no vertex of $W_1$ is adjacent to $(x_2,y_2)$ because $xx_2 \notin E(G)$. Let $W_2$ be a walk $x,p_{k+1},\ldots,p_l,x_2$ in the layer $G^{y}$ and $W_3$ a walk $y,q_{a+1},\ldots,q_{b},y_2$ in the layer  $^{x_2}\!H$. If $yy_2 \in E(H)$ define $W_2'$ to be a walk $(x,y), (p_{k+1},y),\ldots,(p_l,y),(x_2,y_2)$. If $yy_2 \notin E(H)$, a concatenation of $W_1,W_2$ and $W_3$ is a desired tolled walk, otherwise a concatenation of $W_1$ and $W_2'$ is a walk which completes the proof. \qed
\end{proof}

From Lemmas \ref{l:main1}, $\ldots ,$ \ref{l:main4} it follows that the toll interval between two vertices $(x_1,y_1),$ $(x_2,y_2) \in V(G \boxtimes H)$, for which $x_1x_2 \notin E(G)$ and $y_1y_2 \notin E(H)$,  covers all vertices of the product except some of their neighbors:

\begin{corollary}\label{p:1}
Let $(x_1,y_1),(x_2,y_2) \in V(G \boxtimes H)$, where $x_1x_2\notin{E(G)}$, $y_1y_2 \notin E(H)$. If $(x,y)\in V(G \boxtimes H)\setminus \left( N((x_1,y_1)) \cup N((x_2,y_2 ))\right)$, then $(x,y) \in T_{G \boxtimes H}\left((x_1,y_1),(x_2,y_2)\right)$.
\end{corollary}

%
%
%

 A toll number was already studied in the Cartesian and the lexicographic product graphs~\cite{gr}. In both cases, the exact results on the number of extreme vertices are not known (except the bounds that holds in all graph classes). In the rest of this section we will consider extreme vertices of the strong product graphs.
Using the results, listed above, we will prove that there are no extreme vertices with respect to toll convexity in the strong product of two non complete graphs.

\begin{lemma}
Let $x \in V(G)$ be a vertex, that is not an extreme vertex in $G$. Then, for any $y \in V(H)$, $(x,y)$ is not an extreme vertex of the product $G \bt H$.
\end{lemma}

\begin{proof}
As $x$ is not an extreme vertex in $G$, there are $x_1,x_2 \in V(G)\setminus \lbrace x \rbrace$ such that $x \in T_G(x_1,x_2)$. Denote with $W$ an $x_1,x_2$-tolled walk that contains $x$. Then $W$ in the layer $G^y$ is a tolled $(x_1,y),(x_2,y)$-walk that contains $(x,y)$. Therefore 
$(x,y)$ is not an extreme vertex of $G \bt H$.\qed
\end{proof}

\begin{corollary}\label{p:ext}
Let $G$ and $H$ be connected graphs, that are not complete graphs. If a vertex $(x,y) \in V(G \boxtimes H)$ is an extreme vertex, then $x$ is extreme in $G$ and $y$ is extreme in $H$.
\end{corollary}

\begin{theorem}
Let $G$ and $H$ be connected graphs, that are not complete graphs. Then a graph $G \boxtimes H$ has no extreme vertices.
\end{theorem}

\begin{proof}
Assume $(x,y) \in V(G \bt H)$ is an extreme vertex of $G \bt H$. By Corollary \ref{p:ext}, $x$ is extreme in $G$ and $y$ is extreme in $H$. Since any extreme vertex is simplicial, $x$ is simplicial in $G$ and $y$ is simplicial in $H$, but none of them is a universal vertex as $G$ and $H$ are not complete graphs. Let $x_1 \in ecc_G(x)$ and $y_1 \in ecc_H(y)$. Note that $d_G(x,x_1)\geq 2$ and $d_H(y,y_1)\geq 2$. Therefore, by Corollary \ref{p:1}, $(x,y) \in T_{G \bt H}((x_1,y),(x,y_1))$, a contradiction.\qed
\end{proof}

\section{Toll number of $G \bt H$}\label{s:tn}

In this section, the toll number of the strong product of two graphs $G$ and $H$ will be considered. Assume first that at least one of  $G$ and $H$  is a complete graph. Observe that  $G\boxtimes K_n \cong G \circ K_n$. As the exact formula for the toll number of $G \circ K_n$ was obtained in~\cite{gr} in terms of toll-dominating pairs, the same result holds also for $tn(G \bt K_n)$.  Therefore we may assume from now on that $G$ and $H$ are not complete graphs. We will prove that the toll number of the strong product of two non complete graphs $G$ and $H$ is at most 3. Then we will also characterize graphs $G \bt H$ with toll number 2 and 3.

It is easy to find graphs $G,H$ such that $tn(G \bt H)=2$. Consider the strong product of two paths $P_n$ and $P_m$, where $n,m \geq 3$. Denote $V(P_n) = \{1,2,\ldots , n\}$, $E(P_n)=\{12,23,\ldots (n-1)n\}$, $V(P_m) = \{1,2,\ldots , m\}$, $E(P_m)=\{12,23,\ldots (m-1)m\}$. Then $\{(1,1), (n,m)\}$ is a toll set of $P_n \bt P_m$. On the other hand, $\{(1,1), (n,m)\}$ and $\{(1,m), (n,1)\}$ are not the only toll sets of $P_n \bt P_m$. For example $\{(1,1), (n,1)\}$ is also a toll set of $P_n \bt P_m$. Moreover, it is also possible to construct a toll set $\{(i_1,j_1),(i_2,j_2)\}$ of $G \bt H$ with $d(i_1,i_2) < diam(G)$ and  $d(j_1,j_2) < diam (H)$. In this sense, any set $\{(i,j),(i+k,j+l)~|~i,j \geq 3; k,l \geq 2; i+k \leq n-2; j+l \leq m-2\}$ is a toll set of $P_n \bt P_m$.

There are also graphs $G,H$, with $tn(G \bt H) > 2$. Let $G$ be a $K_3$ together with a pendant vertex, i.e.\ $V(G)=\{g_1,g_2,g_3,g_4\}, E(G)=\{g_1g_2,g_1g_3,g_2g_3,g_3g_4\}$ and let $H$ be isomorphic to $G$ (see Figure~\ref{e:tn3} for graphs $G,H$ and $G \bt H$). Let $A$ be a smallest toll set of $G \bt H$. Note that $N[g_1]=N[g_2] \subset N[g_3]$. Suppose that $(g_i,h_j) \in A$ for some $i \in \{1,2,3\}, j \in \{1,2,3,4\}$. 
Then for any $l\in \{1,2\}-\{i\}$, $(g_l,h_j) \notin T((g_i,h_j),(g,h))$ for any $(g,h) \in V(G \bt H)$. Hence $|A| \geq 3$. By the symmetry of the strong product, if $(g_i,h_j) \in A$ for some $i \in \{1,2,3,4\},  j \in \{1,2,3\}$, then $|A| \geq 3$. Let $B=\{(g_i,h_j)~|~i \in \{1,2,3\}, j \in \{1,2,3,4\}\} \cup \{(g_i,h_j)~|~i \in \{1,2,3,4\}, j \in \{1,2,3\}\}.$ Since $|A \cap B| \geq 1$, the above explanation implies that $|A| \geq 3$. Observe that $\{(g_1,h_1),(g_4,h_1),(g_4,h_4)\}$ is a toll set of $G \bt H$. Therefore $tn(G \bt H)=3$.

\begin{figure}[h]
\centering
\begin{tikzpicture}[scale=1]
	
	\foreach \y in {-1,0,1,2,3}
	{
		\draw (0,\y)--(3,\y);
		\draw (0,\y) to [bend right=20] (2,\y);
	}
	\foreach \x in {-1,0,1,2,3}
	{
		\draw (\x,0)--(\x,3);
		\draw (\x,0) to [bend right=20] (\x,2);
	}
	
	\draw (0,0)--(3,3);
	\draw (0,3)--(3,0);
	\draw (1,0)--(3,2);
	\draw (2,0)--(3,1);
	\draw (2,0)--(0,2);
	\draw (1,0)--(0,1);
	\draw (0,1)--(2,3);
	\draw (3,1)--(1,3);
	\draw (3,2)--(2,3);
	\draw (0,2)--(1,3);
	
	\foreach \x in {0,1,2}
	{
		\draw (\x,0)--(\x+1,2);
		\draw (\x+1,0)--(\x,2);
	}
	\foreach \y in {0,1,2}
	{
		\draw (0,\y)--(2,\y+1);
		\draw (0,\y+1)--(2,\y);
	}
	\draw (0,0) to [bend right=20] (2,2);
	\draw (0,2) to [bend right=20] (2,0);

\foreach \x  in {0,1,2,3}
\foreach \y  in {0,1,2,3}
{
		\filldraw [fill=black, draw=black,thick] (\x,\y) circle (3pt);
}

\foreach \x  in {0,1,2,3}
{
		\filldraw [fill=black, draw=black,thick] (\x,-1) circle (3pt);
}
\foreach \y  in {0,1,2,3}
{
		\filldraw [fill=black, draw=black,thick] (-1,\y) circle (3pt);
}

\node[] at (0,-1.35) {$g_1$};
\node[] at (1,-1.35) {$g_2$};
\node[] at (2,-1.35) {$g_3$};
\node[] at (3,-1.35) {$g_4$};
\node[] at (1.5,-1.8) {$G$};
\node[] at (-1.35,0) {$h_1$};
\node[] at (-1.35,1) {$h_2$};
\node[] at (-1.35,2) {$h_3$};
\node[] at (-1.35,3) {$h_4$};
\node[] at (-1.8,1.5) {$H$};

\end{tikzpicture}
\caption{A graph $G \bt H$ with $tn(G \bt H)=3$.\label{e:tn3}}
\end{figure}

For a path $P=x_1,x_2,\ldots, x_n$ let $P^{-1}$ denote an $x_n,x_1$-path $x_n,x_{n-1},\ldots , x_1$.

The  following theorem says, that the toll number of $G \bt H$ can not exceed 3 in the case of two not complete graphs $G$ and $H$.

\begin{theorem}\label{atMost3}
Let $G$ and $H$ be connected non complete graphs. Then $tn(G \boxtimes H) \leq 3$.
\end{theorem}

\begin{proof}
Let $x_1, x_2$ be diametral vertices of $G$ and $y_1,y_2$ diametral vertices of $H$. Since $G$ and $H$ are not isomorphic to complete graphs, $d_G(x_1,x_2) \geq 2$ and $d_H(y_1,y_2) \geq 2$. We will show that a set $\lbrace (x_1,y_1),(x_2,y_1),(x_2,y_2) \rbrace$ is a toll set of $G \bt H$.

Let $(x,y)$ be an arbitrary vertex of $G \bt H$. If $(x,y)(x_1,y_1),(x,y)(x_2,y_2) \notin E(G \bt H)$, then Corollary~\ref{p:1} implies that $(x,y) \in T_{G\bt H}((x_1,y_1),(x_2,y_2))$. Therefore we may assume that $(x,y)$ is either adjacent to $(x_1,y_1)$ or $(x_2,y_2)$ in $G \bt H$. Without loss of generality, let $(x,y)(x_1,y_1) \in E(G\bt H)$. If $x \in T_G(x_1,x_2)$ and $y \in T_G(y_1,y_2)$, then it follows from Lemma~\ref{l:main1} that $(x,y) \in T_{G \bt H}((x_1,y_1),(x_2,y_2))$.

For the rest of the proof, denote by $P$ a shortest $x_1,x_2$-path in $G$ and by $Q$ a shortest $y_1,y_2$-path in $H$.

Suppose now that $x \notin T_G(x_1,x_2)$, which implies that $x \notin \{x_1,x_2\}$, and $y \in T_H(y_1,y_2)$. 
Since $(x,y)$ is adjacent to $(x_1,y_1)$ and $x \neq x_1$, $xx_1 \in E(G)$. As $x \notin T_G(x_1,x_2)$, $xx_2 \notin E(G)$. If $y \neq y_1$, then Lemma~\ref{l:main2} implies that $(x,y) \in T_{G\bt H}((x_1,y_1),(x_2,y_2))$. Hence let $y=y_1.$  
Then the concatenation $W$  of paths $P^{-1}$ in $G^{y_1}$, a path $(x_1,y_1),(x,y),(x_1,y_1)$, a path $Q$ in $^{x_1}H$ and a path $P$ in $G^{y_2}$ is a tolled $(x_2,y_1),(x_2,y_2)$-walk that contains $(x,y)$. 

If $x \notin T_G(x_1,x_2),y \notin T_H(y_1,y_2)$, then $x \neq x_1$ and $y\neq y_1$. Since $(x,y)$ is adjacent to $(x_1,y_1)$, $x_1x \in E(G)$ and $y_1y \in E(H)$. As $x \notin T_G(x_1,x_2)$ and  $y\notin T_H(y_1,y_2)$, $xx_2 \notin E(G)$ and $yy_2 \notin E(H)$. Then also the concatenation $W$ 
 is a tolled $(x_2,y_1),(x_2,y_2)$-walk that contains $(x,y)$. 

Finally, let $x \in T_G(x_1,x_2)$ and $y\notin T_H(y_1,y_2)$, which implies that $y \notin \{y_1,y_2\}$. Since $(x,y)$ is adjacent to $(x_1,y_1)$ and $y \neq y_1$, $yy_1 \in E(G)$. As $y \notin T_G(y_1,y_2)$, $yy_2 \notin E(G)$. If $x \neq x_1$, then Lemma~\ref{l:main3} implies that $(x,y) \in T_{G\bt H}((x_1,y_1),(x_2,y_2))$. Hence let $x=x_1.$ Then, again, the concatenation $W$
form a tolled $(x_2,y_1),(x_2,y_2)$-walk that contains $(x,y)$.  \qed
\end{proof}

The rest of this section will be focused in finding  a characterization of graphs $G \bt H$ with toll number 2. First, there is a necessary condition for graphs $G \bt H$ having toll number 2. 

\begin{theorem}\label{th:nec}
Let $G$ and $H$ be connected non complete graphs. If $tn(G \bt H)=2$, then there exist two different, non adjacent vertices $(x_1,y_1),(x_2,y_2) \in V(G \bt H)$ such that
\begin{enumerate}
\item for any $(x,y) \in N((x_1,y_1))$, $N\left[(x,y)\right] \nsubseteq N\left[(x_1,y_1)\right]$ and
\item for any $(x,y) \in N((x_2,y_2))$, $N\left[(x,y)\right] \nsubseteq N\left[(x_2,y_2)\right]$.
\end{enumerate}
\end{theorem} 
\begin{proof}
For the purpose of contradiction suppose that for any non adjacent vertices $(x_1,y_1),$ $(x_2,y_2) \in V(G \bt H)$, there exists $(x,y) \in N((x_1,y_1))$ with $N\left[ (x,y) \right] \subseteq N\left[  (x_1,y_1) \right]$ or there exists $(x,y) \in N((x_2,y_2))$ with $N\left[ (x,y) \right] \subseteq N\left[  (x_2,y_2) \right]$.
 
Let $(x_1,y_1),(x_2,y_2) \in V(G \bt H)$ be arbitrary non adjacent vertices. Without loss of generality assume that there exists $(x,y) \in N((x_1,y_1))$ with $N\left[ (x,y) \right] \subseteq N\left[  (x_1,y_1) \right]$. Since any neighbor of $(x,y)$ is also neighbor of $(x_1,y_1)$, $(x,y) \notin T_{G\bt H}((x_1,y_1),(x_2,y_2))$. Therefore $\{(x_1,y_1),(x_2,y_2)\}$ is not a toll set for any non adjacent vertices $(x_1,y_1),(x_2,y_2)$. As two adjacent vertices do not generate a toll set, $tn(G \bt H) >2$, a contradiction.   \qed
\end{proof}

Before we prove that the condition from Theorem~\ref{th:nec} is also a sufficient condition for a strong product graph to have toll number 2, we need the following lemmas. Results in Section~\ref{interval} are about toll intervals between two non adjacent vertices of the strong product when the projections of both vertices on both factors are also non adjacent. Next results deal with the remaining cases of non adjacent vertices of the strong product graphs ($y_1=y_2$ or $y_1y_2 \in E(H)$)  and hold if additional condition is added.

\begin{lemma}\label{l:equal}
Let $G$ and $H$ be connected non complete graphs. Suppose that there exist two different, non adjacent vertices $(x_1,y_1),(x_2,y_1) \in V(G \bt H)$, such that
\begin{enumerate}
\item for any $(x,y) \in N((x_1,y_1))$, $N\left[(x,y)\right] \nsubseteq N\left[(x_1,y_1)\right]$ and
\item for any $(x,y) \in N((x_2,y_1))$, $N\left[(x,y)\right] \nsubseteq N\left[(x_2,y_1)\right]$.
\end{enumerate}
Then, for any $(x,y) \notin  N\left((x_1,y_1) \right) \cup N\left((x_2,y_1\right))$, $(x,y) \in T_{G \bt H}((x_1,y_1),(x_2,y_1))$.
\end{lemma}

\begin{proof}
First assume that $y_1$ is an universal vertex of $H$. Then for any $y \in V(H) \setminus \{y_1\}$ it holds that $(x_1,y)(x_1,y_1)\in E(G \boxtimes H)$ and $N[(x_1,y)] \subseteq N[(x_1,y_1)]$, a contradiction. Therefore there exists $y' \in V(H)$ such that $y_1y' \notin E(H)$. It is also clear that $x_1\notin N[x_2]$, as $(x_1,y_1),(x_2,y_1)$ are different and non adjacent. 

Let $(x,y) \in V(G \bt H) \setminus (N\left((x_1,y_1) \right) \cup N\left((x_2,y_1\right)))$. 
Let $P$ be a shortest $y_1,y$-path in $H$, $P'$ a shortest $y_1,y'$-path in $H$, $P''$ a shortest $y',y$-path in $H$ and $Q$ a shortest $x_1,x_2$ walk that contains $x$ in $G$.

If $x \notin N[x_1] \cup N[x_2]$, then the concatenation of $P'$ in $^{x_1}H$, $x_1,x$-subwalk of $Q$ in $G^{y'}$, $P''$ in $^xH$, $P''^{-1}$ in $^xH$, $x,x_2$-subwalk of $Q$ in $G^{y'}$ and $P'^{-1}$ in $^{x_2}H$ is a tolled $(x_1,y_1),(x_2,y_1)$ walk that contains $(x,y)$. 

Now let $x \in N[x_1] \cup N[x_2]$. Without loss of generality let $x \in N[x_1]$. Since $(x,y) \notin N((x_1,y_1))$, $y \notin N[y_1]$ (note that the case when $(x,y)=(x_1,y_1)$ is trivial). Then the concatenation of $P$ in $^{x_1}H$, $Q$ in $G^{y}$ and $P^{-1}$ in $^{x_2}H$ is a tolled $(x_1,y_1),(x_2,y_1)$ walk that contains $(x,y)$.  \qed
\end{proof}

\begin{lemma}\label{l:neighbors}
Let $G$ and $H$ be connected non complete graphs. Suppose that there exist two different, non adjacent vertices $(x_1,y_1),(x_2,y_2) \in V(G \bt H)$, where $y_1y_2 \in E(H)$ such that
\begin{enumerate}
\item for any $(x,y) \in N((x_1,y_1))$, $N\left[(x,y)\right] \nsubseteq N\left[(x_1,y_1)\right]$ and
\item for any $(x,y) \in N((x_2,y_2))$, $N\left[(x,y)\right] \nsubseteq N\left[(x_2,y_2)\right]$.
\end{enumerate}
Then, for any $(x,y) \notin  N\left((x_1,y_1) \right) \cup N\left((x_2,y_2\right))$, $(x,y) \in T_{G \bt H}((x_1,y_1),(x_2,y_2))$.
\end{lemma}
\begin{proof}
Since $(x_1,y_1),(x_2,y_2)$ are not adjacent, $x_1 \notin N[x_2]$.
Let $(x,y) \in V(G \bt H) \setminus (N\left((x_1,y_1) \right) \cup N\left((x_2,y_2\right)))$. If $x \in T_G(x_1,x_2)$, then one can easily check that $(x,y) \in T_{G \bt H}((x_1,$ $y_1),(x_2,y_2))$. Therefore we may assume that $x \notin T_G(x_1,x_2)$ and consequently $x \notin \{x_1,x_2\}$.  We will distinguish two cases.

Assume first that $x \notin N(x_1) \cup N(x_2).$ If $N\left[y_1\right] \subseteq N\left[y_2\right]$, this would imply that $N\left[(x_2,y_1)\right] \subseteq N\left[(x_2,y_2)\right]$, which can not be true by assumption number 2 of this theorem. Therefore $N\left[y_1\right] \nsubseteq N\left[y_2\right]$. Similar argument implies that $N\left[y_2\right] \nsubseteq N\left[y_1\right]$. Therefore, there exist $y' \in V(H)$, such  that $y'y_1\in E(H)$ but $y'y_2\notin E(H)$ and there exists $y''\in V(H)$, such that $y''y_2 \in E(H)$, but $y''y_1\notin E(H)$.  Let $P=x_1,t_1,\ldots , t_{k-1},x$ be a shortest $x_1,x$-path in $G$ and let $Q=y',z_2,\ldots , z_{l-1},y$ be a shortest $y',y$-path in $H$ and let $R=x,u_2,\ldots , u_{j-1},x_2$ be a shortest $x,x_2$-path in $G$. Then the following walk is a tolled $(x_1,y_1),(x_2,y_2)$-walk that contains $(x,y)$. Start from $(x_1,y_1)$ to $(t_1,y')$, then take $t_1,x$-subpath of $P$ in $G^{y'}$, follow  $z_2,y$-subpath of $Q$ in $^{x}\!H$, continue with a shortest $y,y''$-path in $^{x}\!H$ and finally follow $x,u_{j-1}$-subpath of $R$ in $G^{y''}$ and finish in $(x_2,y_2)$.

Suppose now that $x \in N(x_1) \cup N(x_2)$. Since $x \notin T_G(x_1,x_2)$, $x$ can not be adjacent to both $x_1$ and $x_2$. Without loss of generality assume that $xx_1 \in E(G)$ and $xx_2 \notin E(G)$. 
Note that $y\notin N_H[y_1]$, as $(x,y)(x_1,y_1) \notin E(G \bt H)$. If $yy_2 \notin E(H)$, then a tolled $(x_1,y_1),(x_2,y_2)$-walk that contains $(x,y)$ can be constructed in the following way. 
Start the walk with a shortest $y_1,y$-path in $^{x_1}\!H$. Then follow with $(x,y)$ and a shortest $x,x_2$-path in $G^{y}$. Continue with a shortest $y,y_2$-path in $^{x_2}\!H$. If $yy_2 \in E(G)$, then a tolled $(x_1,y_1),(x_2,y_2)$-walk that contains $(x,y)$ can be constructed in the following way. Start the walk with $(x_1,y_1),(x_1,y_2),(x_1,y),(x,y),(x_1,y)$ then take a shortest path to $(a,y)$ in $G^y$, where $a$ is a neighbor of $x_2$ on a shortest $x_1,x_2$-path in $G$. Finish the walk with $(x_2,y_2)$. \qed
\end{proof}

By Lemma \ref{l:equal}, Lemma \ref{l:neighbors} and by the symmetry of the strong product we have the following.

\begin{lemma}\label{l:neighbors2}
Let $G$ and $H$ be connected non complete graphs. Suppose that there exist two different, non adjacent vertices $(x_1,y_1),(x_2,y_2) \in V(G \bt H)$, where $x_1 \in N[x_2]$ such that
\begin{enumerate}
\item for any $(x,y) \in N((x_1,y_1))$, $N\left[(x,y)\right] \nsubseteq N\left[(x_1,y_1)\right]$ and
\item for any $(x,y) \in N((x_2,y_2))$, $N\left[(x,y)\right] \nsubseteq N\left[(x_2,y_2)\right]$.
\end{enumerate}
Then, for any $(x,y) \notin  N\left((x_1,y_1) \right) \cup N\left((x_2,y_2\right)),$ $(x,y) \in T_{G \bt H}((x_1,y_1),(x_2,y_2))$.
\end{lemma}

\begin{theorem}\label{th:suf}
Let $G$ and $H$ be connected non complete graphs. If there exist two different non adjacent vertices $(x_1,y_1),(x_2,y_2) \in V(G \bt H)$ such that
\begin{enumerate}
\item for any $(x,y) \in N((x_1,y_1))$, $N\left[(x,y)\right] \nsubseteq N\left[(x_1,y_1)\right]$ and
\item for any $(x,y) \in N((x_2,y_2))$, $N\left[(x,y)\right] \nsubseteq N\left[(x_2,y_2)\right]$,
\end{enumerate}
then $tn(G \bt H)=2$.
\end{theorem}

\begin{proof}
We will prove that $\{(x_1,y_1),(x_2,y_2)\}$ is a toll set of $G \bt H$. Let $(x,y) \in V(G\bt H)$. If $(x,y)$ is not adjacent to neither $(x_1,y_1)$ nor $(x_2,y_2)$, then it follows from Lemmas~\ref{l:equal},~\ref{l:neighbors},~\ref{l:neighbors2} and Corollary~\ref{p:1} that $(x,y) \in T_G((x_1,y_1),(x_2,y_2))$. Therefore let $(x,y) \in N\left((x_1,y_1) \right) \cup N\left((x_2,y_2\right))$. 

Without loss of generality let $(x,y) \in N((x_1,y_1))$. Also, assume that $(x,y) \notin N((x_2,y_2))$. Note that if $(x,y) \in N((x_2,y_2))$, then $(x_1,y_1),(x,y),(x_2,y_2)$ is a tolled walk and we are done. 
By the assumption of the theorem, there exists $(x',y') \in V(G \bt H)$, which is adjacent to $(x,y)$, but it is not adjacent to $(x_1,y_1).$ Therefore $x' \notin N_G \left[ x_1 \right]$ or $y' \notin N_H \left[ y_1 \right]$. Without loss of generality we may assume that $x' \notin N_G \left[ x_1 \right]$. Let $P=x',t_1,\ldots , t_k,x_2$ be a shortest $x',x_2$-path in $G$ and $Q=y',z_1,\ldots , z_l,y_2$ a shortest $y',y_2$-path in $H$.

First assume that $y_1\notin N_H[y_2]$. Then a tolled walk between $(x_1,y_1),(x_2,y_2)$ that contains $(x,y)$ is the following. Start with $(x_1,y_1),(x,y),(x',y')$. If $(x',y')(x_2,y_2) \in E(G \bt H)$ then finish this walk in $(x_2,y_2)$, otherwise continue the walk with $z_1,z_l$-subpath of $Q$ in $^{x'}\!H$ and then with $t_1,x_2$-subpath of $P$ in $G^{y_2}$.

Finally, let $y_1 \in N_H[y_2]$. Since $(x_1,y_1),(x_2,y_2)$ are different and non adjacent $x_1\notin N_G[x_2]$. If $y' \notin N_H[y_1]$, then a desired walk is the following. Start with $(x_1,y_1),(x,y),(x',y')$. If $(x',y')(x_2,y_2) \in E(G \bt H)$, then finish this walk in $(x_2,y_2)$, otherwise continue the walk with $t_1,t_k$-subpath of $P$ in $G^{y'}$ and then with $z_1,y_2$-subpath of $Q$ in $^{x_2}H$.

To conclude the proof, let $y' \in N_H\left[ y_1 \right]$. 
If $y_1y_2 \in E(H)$, $N\left[y_2\right] \subseteq N\left[y_1\right]$ would imply $N\left[(x_1,y_2)\right] \subseteq N\left[(x_1,y_1)\right]$ (which is contrary to our assumption), thus $N\left[y_2\right] \nsubseteq N\left[y_1\right]$. Therefore, there exists  $y'' \in V(H)$ that is adjacent to $y_2$ but it is not adjacent to $y_1$. If $y_1=y_2$, then the assumption of the theorem implies that $N[y_1] \neq V(H)$. In this case let $y''$ be an arbitrary vertex at distance 2 from $y_1$ and let $y_1,a,y''$ be a shortest $y_1,y''$-path in $H$.
Let $R$ be a shortest $y',y''$-path that contains $y_2$.
Then a tolled $(x_1,y_1),(x_2,y_2)$-walk that contains $(x,y)$ is the following. Start with $(x_1,y_1),(x,y),(x',y')$. If $x'x_2 \in E(G)$ then finish this walk following a shortest $y',y_2$-path (starting in the neighbor of $y'$) in $^{x_2}H$ (note that when $y_1=y_2$ this path contains just vertex $(x_2,y_1)$). 
Otherwise, continue the walk following the path $R$ in $^{x'}\!H$, then follow $t_1,t_k$-subpath of $P$ in $G^{y''}$ and finish with $(x_2,y_2)$ if $y_1y_2 \in E(H)$ and with $(x_2,a),(x_2,y_1)$ if $y_1=y_2$.  \qed
\end{proof}

The proof of Theorem \ref{th:nec} and Theorem \ref{th:suf} give a characterization of strong product graphs with $tn(G \bt H)=2$.

\begin{corollary}
Let $G$ and $H$ be connected non complete graphs. Then  $tn(G \bt H)=2$ if and only if there exist two different non adjacent vertices $(x_1,y_1),(x_2,y_2) \in V(G \bt H)$ such that
\begin{enumerate}
\item for any $(x,y) \in N((x_1,y_1))$, $N\left[(x,y)\right] \nsubseteq N\left[(x_1,y_1)\right]$ and
\item for any $(x,y) \in N((x_2,y_2))$, $N\left[(x,y)\right] \nsubseteq N\left[(x_2,y_2)\right]$,
\end{enumerate}
\end{corollary}

Together with Theorem~\ref{atMost3} we also have a characterization of strong product graphs with toll number 3.

\section{Concluding remarks}\label{s:th}
In previous sections, the toll number of the strong product graphs was considered. Since toll number of a graph is an upper bound for the t-hull number of a graph, it is clear that $th(G \bt H) \leq 3$ for any non complete graphs $G$ and $H$.
t-hull number of the Cartesian product and the lexicographic product of non complete graphs was obtained in~\cite{gr} as an immediate consequence of the characterization of t-convex sets in those two graph products proved in~\cite{abg}. It was shown that $th(G \Box H)=th(G \circ H)=2$ (when $G$ and $H$ are not complete), since proper t-convex sets in the Cartesian product exist just if one factor is a complete graph and proper t-convex sets in the lexicographic product exist just if the second factor is a complete graph. In the case of the strong product, proper t-convex sets can also exist when both factors are non complete graphs. Therefore t-hull number can not be obtained in the same way as in the case of the Cartesian and the lexicographic product. Anyway, $th(G \bt H)=2$ if $G$ and $H$ are non complete graphs.

\begin{theorem}
Let $G$ and $H$ be connected non complete graphs. Then $th(G \boxtimes H) = 2$.
\end{theorem}    
 
\begin{proof}
Let $x_1, x_2$ be diametral vertices of $G$ and $y_1,y_2$ diametral vertices of $H$. Since $G$ and $H$ are not isomorphic to complete graphs, $d_G(x_1,x_2) \geq 2, d_H(y_1,y_2) \geq 2$.
Since any toll set is also t-hull set, it follows from the proof of Theorem~\ref{atMost3} that $\{(x_1,y_1),(x_2,y_1),(x_2,y_2)\}$ is a t-hull set of $G \bt H$. As $(x_2,y_1) \in T_{G \bt H}((x_1,y_1),(x_2,y_2))$, $\{(x_1,y_1),(x_2,y_2)\}$ is a t-hull set of $G \bt H$. \qed
\end{proof}


\section*{Acknowledgements}
Research of T.\ Gologranc was supported by Slovenian Research Agency under the grants N1-0043 and  P1-0297.


\end{document}